\documentclass[reqno, 12pt]{amsart}
\usepackage{enumerate}
\usepackage{amssymb,url,color}
\usepackage{hyperref}
\usepackage{amsmath,amsthm}

\allowdisplaybreaks[4]

\newtheorem{thm}{Theorem}[section]
\newtheorem{cor}{Corollary}[section]
\newtheorem{prop}{Proposition}[section]
\newtheorem{lem}{Lemma}[section]
\newtheorem{rem}{Remark}[section]
\newtheorem{exa}{Example}[section]
\theoremstyle{Problem}

\theoremstyle{definition}
\newtheorem{defn}{Definition}[section]
\numberwithin{equation}{section}
\newcommand{\pp}{\mathbb{P}}

\newcommand{\ee}{\mathbb{E}}
\newcommand{\bb}{\mathcal{B}}
\newcommand{\FF}{\mathcal{F}}
\newcommand{\dd}{\mathcal{D}}
\newcommand{\rr}{\mathbb{R}}

\def\beq{\begin{equation}}
\def\deq{\end{equation}}
\allowdisplaybreaks 

\bfseries\rmfamily 

\begin{document}
\title[MDP for plug-in estimators]
{Moderate deviation principle for plug-in estimators of
diversity indices on countable alphabets}
\thanks{This work is supported by National Natural Science Foundation of China (NSFC-11971154).}

\author[Z. H. Yu]{Zhenhong Yu}
\address[Z. H. Yu]{School of Mathematics and Statistics, Henan Normal University, Henan Province, 453007, China.} \email{\href{mailto: Z. H. Yu
<zhenhongyu2022@126.com>}{zhenhongyu2022@126.com}}

\author[Y. Miao]{Yu Miao}
\address[Y. Miao]{School of Mathematics and Statistics, Henan Normal University, Henan Province, 453007, China.} \email{\href{mailto: Y. Miao
<yumiao728@gmail.com>}{yumiao728@gmail.com}; \href{mailto: Y. Miao <yumiao728@126.com>}{yumiao728@126.com}}

\begin{abstract}
In the present paper, we consider the moderate deviation principle for the plug-in estimators of a large class of diversity indices on countable alphabets, where the distribution may change with the sample size. Our results cover some of the most commonly used indices, including Tsallis entropy, Re\'{n}yi
entropy and Hill diversity number.
\end{abstract}

\maketitle
\section{Introduction}
Diversity is a fundamental concept across numerous scientific disciplines. Historically, the interest stems from ecological applications, where the diversity of species in an ecosystem is a relevant issue. Other applications include cancer research, where the interest is in the diversity of types of cancer cells in a tumour, and linguistics, where it is in the diversity of an author's vocabulary. More generally, in information science, one is interested in the diversity of letters drawn from some alphabet. A diversity index is a measures of the amount of variability or randomness in a probability distribution on an alphabet. where there is no natural ordering and moments, such as variance and standard deviation, are not defined. Two of the earliest diversity
indices to appear in the literature are Shannon's entropy and Simpson's index. Since then, many indices have been developed.
 Patil and Taillie \cite{P-T} presented an up-to-date description of different approaches to diversity, a concept whose usage from ecology to linguistics, from economics to genetics is known. Grabchak et al. \cite{G} introduced the generalized Simpson's entropy as a measure of diversity and investigate its properties. Based on comparing the entropy of the two samples, Grabchak et al. \cite{G-Z-Z}
 proposed a new methodology for testing the authorship of a relatively small work compared with the large body of an author's cannon.
 Grabchak et al. \cite{G-C-Z} gave a new methodology for testing whether two writing samples were written by the same author. More generally, in information science, one is interested in the diversity of letters drawn from some alphabet. Zhang and Grabchak \cite{Z-G} showed that a large class of diversity indices in the literature can be represented by linear combinations of an entropic basis, and proposed a class of nonparametric estimators of such linear diversity indices.

To evaluate a diversity index, the most popular approach may be to use the so-called `plug-in' estimator, where one evaluates the diversity index on the empirical distribution, and thus `plugs' the empirically observed probabilities into the formula for the diversity index. The plug-in estimator is one of the most common and serves as a foundation for constructing further estimators. Therefore, understanding the statistical properties of the plug-in estimator is crucial for comprehending many related estimators. Most work of the plug-in estimator in the area of diversity indices has focused on the case of finite alphabets.
For example, Zhang and Grabchak \cite{Z-G} gave a characterisation of all diversity indices, including those on countably infinite alphabets, the asymptotic properties of the plug-in (and related estimators) are only shown for finite alphabets. In fact, there is relatively little research on the asymptotic properties of interpolation on countably infinite alphabets. The results that we have seen in the literature are specifically related to Shannon's entropy. For which, the asymptotically normal for the plug-in estimator of Shannon's entropy defined on a countable alphabet was proved, in two steps, in Paninski \cite{P} and Zhang and Zhang \cite{ZhZh12}.
Grabchak and Zhang \cite{G-Z} studied the asymptotic
distribution of the plug-in estimator for a large class of diversity indices on countable alphabets.
 In particular, they gave conditions for the plug-in estimator
to be asymptotically normal, and in the case of uniform distributions,
where asymptotic normality fails, they gave conditions for the
asymptotic distribution to be chi-squared. Their results covered some of
the most commonly used indices, including Simpson's index, Re\'nyi's
entropy and Shannon's entropy.

In the present paper, we shall study the moderate deviation principle of  the plug-in estimator for a large class of diversity indices along the work in Grabchak and Zhang \cite{G-Z}.
In Section 2, we sate the main results. In Section 3, we discuss some examples to show that these conditions can be satisfied. In particular, we give the moderate deviation principle for Tsallis entropy, Re\'{n}yi
entropy and Hill diversity number. The proofs of our results will be given in Section 4.

\section{Main results}

Let $\mathcal{A}=\{a_k, k\ge 1\}$ be a countably infinite alphabet with associated probability measures $\mathbf{P}_n=\{p_{n,k}, k\ge 1\}$ for each $n$, where the distribution may change with the sample size. The letters of $\mathcal{A}$ correspond to species in an ecosystem, words in the English language, types of cancer cells in a tumour, or another quantity, whose diversity is of interest. We allow some (even countably many) $p_{n,k}$s to be zero. Thus finite alphabets are a special case of this model.

For each $n$, a diversity index is a function $\theta$ that maps $\mathbf{P}_n$ into $\rr$. A common assumption is that
\beq\label{theta}
\theta_n=\theta(\mathbf{P}_n)=\sum_{i=1}^\infty g\left(p_{n,i}\right),
\deq
where $g: [0,1]\to \rr$. Such indices (under a slightly different parametrisation) were called dichotomous indices in Patil and Taillie \cite{P-T}. To ensure that the index is well defined, we assume that
\beq\label{theta1}
\sum_{i=1}^\infty |g\left(p_{n,i}\right)|<\infty,\ \ \ \text{for each}\ n.
\deq
For each $n\ge 1$, let $\{X_{k,n}, 1\le k\le n\}$ be an array of independent and identically distributed random variables taking values in some countably infinite alphabet $\mathcal{A}$ with common distribution $\mathbf{P}_n$, i.e.,
$$
p_{n,k}=\pp(X_{1,n}=a_k),  \ \ \  \ k\ge 1, \ n\ge 1.
$$
For each $k$, let
$$
\hat p_{n,k}:=\frac{1}{n}\sum_{i=1}^nI_{\{X_{i,n}=a_k\}}
$$
be the sample proportion. The plug-in estimator of $\theta$ is given by
\beq\label{theta2}
\hat\theta_n=\sum_{k=1}^\infty g(\hat p_{n,k}).
\deq
To state our main results, we need the following definition.
\begin{defn}\label{def1}
Fix $\beta\in(0,1]$. A function $g:[0,1] \mapsto \rr$ is called $\beta$-H\"older continuous if there is a constant $K > 0$ such that, for any $x,y \in [0,1]$, we have
$$
|g(x)-g(y)|\leq K|x-y|^\beta.
$$
\end{defn}
It is easy to see that every $\beta$-H\"older continuous function is continuous and bounded on a closed interval. It is well known that $1$-H\"older continuous function is also called Lipschitz continuous function, and any function with a bounded derivative is Lipschitz continuous.

Firstly, we consider the case that $g'$ is Lipschitz continuous.
\begin{thm}\label{thm1}
 Suppose the function $g:[0,1] \mapsto \rr$ is differentiable and its derivative $g'$ is Lipschitz
continuous. Let
\beq\label{thm1-1}
\sigma^2_n=\sum_{i=1}^{\infty}p_{n,i}\left(g'\left( p_{n,i}\right)\right)^2-\left(\sum_{i=1}^{\infty}p_{n,i}g'\left( p_{n,i}\right)\right)^2.
\deq
Then for any $r > 0$, we have
\beq\label{thm1-2}
\lim_{n\rightarrow\infty}\frac{1}{b_n^2}\log\pp
\left(\frac{\sqrt{n}}{b_n\sigma_n}|\hat{\theta}_n-\theta_n|>r\right)= -\frac{r^2}{2}
\deq
where the moderate deviation scale $\{b_n, n\ge 1\}$ is a sequence of positive numbers satisfying
\beq\label{con1}
b_n\rightarrow\infty\ \ and\ \ \frac{b_n}{\sqrt{n}\sigma_n}\rightarrow 0.
\deq
\end{thm}

\begin{rem} Since $g'$ is Lipschitz continuous, then $g'$ is also bounded. It is easy to check that $\sigma_n^2<M$ for some positive constant $M$. Furthermore, the theorem allows for the case $\sigma_n^2\to 0$ so long as the convergence is not too fast.
\end{rem}

For this result to be useful for inference, we need a way to estimate $\sigma_n^2$.
\begin{cor}\label{cor2}
Under the assumptions of Theorem \ref{thm1}, let
\beq\label{cor2-1}
\hat\sigma^2_n=\sum_{i=1}^{\infty}\hat p_{n,i}\left(g'\left( \hat p_{n,i}\right)\right)^2-\left(\sum_{i=1}^{\infty}\hat p_{n,i}g'\left( \hat p_{n,i}\right)\right)^2.
\deq
If $\liminf_{n\to\infty} \sigma_n^2>0$, then for any $r > 0$, we have
\beq\label{cor2-2}
\lim_{n\rightarrow\infty}\frac{1}{b_n^2}\log\pp
\left(\frac{\sqrt{n}}{b_n\hat \sigma_n}|\hat{\theta}_n-\theta_n|>r\right)= -\frac{r^2}{2}.
\deq
\end{cor}

Let $\{X_n, n\ge 1\}$ be a sequence of independent and identically distributed random variables taking values in alphabet $\mathcal{A}=\{a_k, k\ge 1\}$ with distribution
$\mathbf{P}=\{p_{n}, n\ge 1\}$ i.e.,
$$
p_{n}=\pp(X_{1}=a_n),  \ \ \  \ n\ge 1.
$$
For each $k$, let
$$
\hat p_{k}:=\frac{1}{n}\sum_{i=1}^nI_{\{X_{i}=a_k\}}
$$
be the sample proportion. The plug-in estimator of $\theta$ is given by
\beq\label{theta2}
\hat\theta_n=\sum_{k=1}^\infty g(\hat p_{k}).
\deq

\begin{thm}\label{cor1}
Suppose the function $g$ is differentiable on $[0,1]$ and its derivative $g'$ is Lipschitz
continuous. Let
\beq\label{cor1-1}
\sigma^2=\sum_{i=1}^{\infty}p_{i}\left(g'\left( p_{i}\right)\right)^2-\left(\sum_{i=1}^{\infty}p_{i}g'\left( p_{i}\right)\right)^2.
\deq\label{cor1-2}
Then for any $r > 0$, we have
\beq\label{cor1-2}
\lim_{n\rightarrow\infty}\frac{1}{b_n^2}\log\pp
\left(\frac{\sqrt{n}}{b_n\sigma}|\hat{\theta}_n-\theta|>r\right)= -\frac{r^2}{2}
\deq
where the moderate deviation scale $\{b_n, n\ge 1\}$ is a sequence of positive numbers satisfying
$$
b_n\rightarrow\infty\ \ \text{and}\ \ \frac{b_n}{\sqrt{n}}\to 0.
$$
\end{thm}

\begin{cor}\label{cor3}
Under the assumptions of Theorem \ref{cor1}, let
\beq\label{cor3-1}
\hat\sigma^2_n=\sum_{i=1}^{\infty}\hat p_{i}\left(g'\left( \hat p_{i}\right)\right)^2-\left(\sum_{i=1}^{\infty}\hat p_{i}g'\left( \hat p_{i}\right)\right)^2.
\deq
If $\sigma^2>0$, then for any $r > 0$, we have
\beq\label{cor3-2}
\lim_{n\rightarrow\infty}\frac{1}{b_n^2}\log\pp
\left(\frac{\sqrt{n}}{b_n\hat \sigma_n}|\hat{\theta}_n-\theta_n|>r\right)= -\frac{r^2}{2}.
\deq
\end{cor}

Next we consider the case that $g'$ is $\beta$-H\"older continuous.
\begin{thm}\label{thm2}
 Suppose the function $g:[0,1] \mapsto \rr$ is differentiable and its derivative $g'$ is $\beta$-H\"older continuous for some $\beta\in(2^{-1}, 1)$. Let
\beq\label{thm2-1}
\sigma^2_n=\sum_{i=1}^{\infty}p_{n,i}\left(g'\left( p_{n,i}\right)\right)^2-\left(\sum_{i=1}^{\infty}p_{n,i}g'\left( p_{n,i}\right)\right)^2.
\deq
Then for any $r > 0$, we have
\beq\label{thm2-2}
\lim_{n\rightarrow\infty}\frac{1}{b_n^2}\log\pp
\left(\frac{\sqrt{n}}{b_n\sigma_n}|\hat{\theta}_n-\theta_n|>r\right)= -\frac{r^2}{2}
\deq
where the moderate deviation scale $\{b_n, n\ge 1\}$ is a sequence of positive numbers satisfying
\beq\label{con11}
b_n\rightarrow\infty\ \ \text{and}\ \ \frac{b_n}{\sqrt{n}\sigma_n^{1/(2\beta-1)}}\to 0.
\deq
\end{thm}

\begin{rem}
It is worth noting that when $\beta=1$, the theorem is precisely Theorem \ref{thm1}. We discuss the Lipschitz continuous case and $\beta$-H\"older continuous case separately, because the proof of the $\beta$-H\"older continuous case is relies on the Lipschitz continuous.
\end{rem}

\begin{rem} Grabchak and Zhang \cite{G-Z} studied the asymptotic normality of $\hat\theta_n$ for the $\beta$-H\"older continuous case with $\beta\in(0,1]$. For the moderate deviation principle of $\hat\theta_n$, we only discuss the case $\beta\in(2^{-1},1]$. At present, it is still impossible to prove the case $\beta\in(0,2^{-1}]$.
\end{rem}

\begin{cor}\label{cor2b}
Under the assumptions of Theorem \ref{thm2}, let
\beq\label{cor2b-1}
\hat\sigma^2_n=\sum_{i=1}^{\infty}\hat p_{n,i}\left(g'\left( \hat p_{n,i}\right)\right)^2-\left(\sum_{i=1}^{\infty}\hat p_{n,i}g'\left( \hat p_{n,i}\right)\right)^2.
\deq
If $\liminf_{n\to\infty} \sigma_n^2>0$, then for any $r > 0$, we have
\beq\label{cor2b-2}
\lim_{n\rightarrow\infty}\frac{1}{b_n^2}\log\pp
\left(\frac{\sqrt{n}}{b_n\hat \sigma_n}|\hat{\theta}_n-\theta_n|>r\right)= -\frac{r^2}{2}.
\deq
\end{cor}

\begin{thm}\label{cor1b}
Suppose the function $g$ is differentiable on $[0,1]$ and its derivative $g'$ is $\beta$-H\"older continuous for some $\beta\in(2^{-1}, 1)$. Let
\beq\label{cor1b-1}
\sigma^2=\sum_{i=1}^{\infty}p_{i}\left(g'\left( p_{i}\right)\right)^2-\left(\sum_{i=1}^{\infty}p_{i}g'\left( p_{i}\right)\right)^2.
\deq\label{cor1b-2}
Let $\hat{\theta}_n$ be defined in (\ref{theta2}). Then for any $r > 0$, we have
\beq\label{cor1b-2}
\lim_{n\rightarrow\infty}\frac{1}{b_n^2}\log\pp
\left(\frac{\sqrt{n}}{b_n\sigma}|\hat{\theta}_n-\theta|>r\right)= -\frac{r^2}{2}
\deq
where the moderate deviation scale $\{b_n, n\ge 1\}$ is a sequence of positive numbers satisfying
$$
b_n\rightarrow\infty\ \ \text{and}\ \ \frac{b_n}{\sqrt{n}}\to 0.
$$
\end{thm}

\begin{cor}\label{cor3b}
Under the assumptions of Theorem \ref{cor1b}, let
\beq\label{cor3b-1}
\hat\sigma^2_n=\sum_{i=1}^{\infty}\hat p_{i}\left(g'\left( \hat p_{i}\right)\right)^2-\left(\sum_{i=1}^{\infty}\hat p_{i}g'\left( \hat p_{i}\right)\right)^2.
\deq
If $\sigma^2>0$, then for any $r > 0$, we have
\beq\label{cor3b-2}
\lim_{n\rightarrow\infty}\frac{1}{b_n^2}\log\pp
\left(\frac{\sqrt{n}}{b_n\hat \sigma_n}|\hat{\theta}_n-\theta_n|>r\right)= -\frac{r^2}{2}.
\deq
\end{cor}

\section{Some examples}
Consider the index
\beq\label{h}
h_{\alpha,\gamma}=\sum_{i=1}^\infty p_i^\alpha\left(1-p_i\right)^\gamma
\deq
for $\alpha > 0$ and $\gamma \geq 0$. When $\alpha = 2$ and $\gamma =0$, this is Simpson's index introduced in Simpson \cite{Simpson}. When $\alpha$ and $\gamma $ are integers, this corresponds to the generalised Simpson's indices introduced in Zhang and Zhou \cite{ZhZh10} and further studied in Grabchak et al. \cite{G}.
When $\alpha> 0$ and $\gamma  = 0$, this corresponds to Re\'{n}yi equivalent entropy introduced in Zhang and Grabchak \cite{Z-G}.

Note that for $h_{\alpha,\gamma}$, where $\alpha > 0$ and $\gamma \geq 0$, we have
$g(x)=x^\alpha(1-x)^\gamma$ and
$$
g'(x)=\alpha x^{\alpha-1}(1-x)^\gamma-\gamma x^\alpha(1-x)^{\gamma-1}.
$$
Furthermore, we recall the following properties.
\begin{prop}\label{prop1} \cite[Proposition 3.1]{G-Z} When $\alpha\ge 1$ and $\gamma\in\{0\}\cup [1,\infty)$, $g'$ is $\beta$-H\"older continuous with
$$
\beta=\begin{cases} \min\{\alpha-1, \gamma-1, 1\}\ \ & \text{if}\ \ \alpha, \gamma >1\\
\min\{\alpha-1,  1\}\ \ & \text{if}\ \ \alpha>1, \gamma\in\{0,1\}\\
\min\{\gamma-1, 1\}\ \ & \text{if}\ \ \alpha=1, \gamma >1\\
1\ \ & \text{if}\ \ \alpha=1,  \gamma\in\{0,1\}\\
\end{cases}.
$$
\end{prop}

\begin{exa}
Consider a sequence of distributions of the form
$$p_{n,1}=\frac{1}{2}+\frac{1}{2n^\gamma}, \ \ \ p_{n,2}=\frac{1}{2}-\frac{1}{2n^\gamma},$$
where $\gamma \in (0,1/2)$ is a real number and $p_{n,i}=0$ for all $i=3,4,\cdots$. Clearly, this approaches a uniform distribution as $n\to \infty$. Suppose that we want to estimate
 Simpson's diversity index, which corresponds to $g(x)=x^2$. In this case, $g'(x)=2x$ is Lipschitz continuous we have
$$
\aligned
\sigma_n^2=&\frac{1}{2}(1+\frac{1}{n^\gamma})^3+\frac{1}{2}(1-\frac{1}{n^\gamma})^3
-\frac{1}{4}\left((1+\frac{1}{n^\gamma})^2+(1-\frac{1}{n^\gamma})^2\right)^2
\\=&\frac{1}{n^{2\gamma}}-\frac{1}{n^{4\gamma}}\sim \frac{1}{n^{2\gamma}}.
\endaligned
$$
If we take $b_n=o(n^{1/2-\gamma})$, then the moderate deviation principle in Theorem \ref{thm1} holds.
\end{exa}
\begin{exa}
For every $i=1,2, \cdots$, let $p_{n,i}=(1-p_n)^{i-1}p_n$ where $p_n=1-\frac{1}{n^\alpha}$ and $\alpha \in (0,1)$. For the case $g(x)=x^2$, we have
$$
\aligned
\sigma_n^2=&4\left(1-\frac{1}{n^\alpha}\right)^3\left(\frac{1}{1-\frac{1}{n^{3\alpha}}}\right)
-4\left(1-\frac{1}{n^\alpha}\right)^4\left(\frac{1}{1-\frac{1}{n^{2\alpha}}}\right)^2\\
=&\frac{4n^\alpha(n^\alpha-1)^2}{(n^{2\alpha}+n^\alpha+1)(n^\alpha+1)^2}
\sim\frac{4}{n^\alpha}.
\endaligned
$$
If we take $b_n=o(n^{1/2-\alpha/2})$, then the moderate deviation principle in Theorem \ref{thm1} holds.
\end{exa}
\begin{exa}
For every $i=1,2, \cdots$, let $p_{i}=C_z i^{-2}$ where
$$
C_{z}=\frac{1}{\sum_{i=1}^\infty i^{-2}}=\frac{1}{\zeta(2)},
$$
$\zeta(s)=\sum_{k=1}^\infty k^{-s}$ is the Riemann zeta function and $\zeta(2)=\frac{\pi^2}{6}$, $\zeta(4)=\frac{\pi^4}{90}$, $\zeta(6)=\frac{\pi^6}{945}$.  For the case $g(x)=x^2$, we have
$$
\aligned
\sigma^2=4C_{z}^3\sum_{i=1}^{\infty}\frac{1}{i^6}-4\left(C_{z}^2\sum_{i=1}^{\infty}\frac{1}{i^4}\right)^2
=4\left(\frac{\zeta(6)}{\zeta(2)^3}-\frac{\zeta(4)^2}{\zeta(2)^4}\right)=\frac{48}{175}.
\endaligned
$$
If we take $b_n=o(n^{1/2})$, then the moderate deviation principles in Theorem \ref{thm1} and Theorem \ref{cor1} hold.
\end{exa}

\begin{exa} Let the index $h_{\alpha,0}$ be defined in (\ref{h}). For $\alpha>1$, consider Tsallis entropy
$$
\mathcal{T}_\alpha:=\frac{1}{1-\alpha} \left(h_{\alpha,0}-1\right)=\frac{1}{1-\alpha}\left(\sum_{k=1}^\infty p_k^\alpha-1\right)
$$
and its plug-in estimator
$$
\hat{\mathcal{T}}_\alpha:=\frac{1}{1-\alpha} \left(\hat h_{\alpha,0}-1\right)=\frac{1}{1-\alpha}\left(\sum_{k=1}^\infty \hat p_k^\alpha-1\right).
$$
Hence we get
\beq\label{remt1}
\hat{\mathcal{T}}_{\alpha,n}-\mathcal{T}_\alpha=\frac{1}{1-\alpha}(\hat h_{\alpha,0}-h_{\alpha,0}).
\deq
Let $g(x)=(1-\alpha)^{-1}x^{\alpha}$ and
$$
\hat\theta_n-\theta:=\sum_{k=1}^{\infty} g(\hat p_k)-\sum_{k=1}^{\infty} g(p_k)=\frac{1}{1-\alpha}(\hat h_{\alpha,0}-h_{\alpha,0}).
$$
From Proposition \ref{prop1}, $g'$ is $\beta$-H\"older continuous with
$\beta=\min\{\alpha-1, 1\}$ for $\alpha>1$.

If $\alpha>1.5$, then $\beta\in(2^{-1}, 1]$.
From Theorem \ref{cor1} and Theorem \ref{cor1b},  for any $r > 0$, we have
\beq\label{remt2}
\lim_{n\rightarrow\infty}\frac{1}{b_n^2}\log\pp
\left(\frac{\sqrt{n}}{b_n\sigma}
|\hat{\mathcal{T}}_{\alpha,n}-\mathcal{T}_\alpha|>r\right)= -\frac{r^2}{2}
\deq
where
 $$
\sigma^2=\left(\frac{\alpha}{\alpha-1}\right)^2\left[\sum_{k=1}^{\infty}p_{k}^{2\alpha-1}
-\left(\sum_{k=1}^{\infty}p_{k}^\alpha\right)^2\right],
$$
and the moderate deviation scale $\{b_n, n\ge 1\}$ is a sequence of positive numbers satisfying
$$
b_n\rightarrow\infty\ \ \text{and}\ \ \frac{b_n}{\sqrt{n}}\to 0.
$$
\end{exa}

\begin{exa}Let the index $h_{\alpha,0}$ be defined in (\ref{h}). For $\alpha>1$, consider  R\'enyi entropy
$$
\mathcal{R}_\alpha:=\frac{1}{1-\alpha}\log h_{\alpha,0}=\frac{1}{1-\alpha}\log\sum_{k=1}^\infty p_k^\alpha
$$
and its plug-in estimator
$$
\hat{\mathcal{R}}_{\alpha,n}:=\frac{1}{1-\alpha}\log \hat h_{\alpha,0}=\frac{1}{1-\alpha}\log\sum_{k=1}^\infty \hat p_k^\alpha.
$$
By using Taylor's formula, we have
$$
\log \hat h_{\alpha,0}=\log h_{\alpha,0}+\frac{\hat h_{\alpha,0}-h_{\alpha,0}}{h_{\alpha,0}}+R_{h_{\alpha,0}}(\hat h_{\alpha,0})
$$
where
$$
R_{h_{\alpha,0}}(\hat h_{\alpha,0}):=-\frac{1}{2\xi^2}(\hat h_{\alpha,0}-h_{\alpha,0})^2
$$
and $\xi$ is between $\hat h_{\alpha,0}$ and $h_{\alpha,0}$. Hence we get
\beq\label{remb1}
\hat{\mathcal{R}}_{\alpha,n}-\mathcal{R}_\alpha=\frac{\hat h_{\alpha,0}-h_{\alpha,0}}{(1-\alpha)h_{\alpha,0}}+\frac{1}{1-\alpha}R_{h_{\alpha,0}}(\hat h_{\alpha,0}).
\deq
Let $g(x)=(1-\alpha)^{-1}h_{\alpha,0}^{-1}x^{\alpha}$ and
$$
\hat\theta_n-\theta:=\sum_{k=1}^{\infty} g(\hat p_k)-\sum_{k=1}^{\infty} g(p_k)=\frac{\hat h_{\alpha,0}-h_{\alpha,0}}{(1-\alpha)h_{\alpha,0}}.
$$
From Proposition \ref{prop1}, $g'$ is $\beta$-H\"older continuous with
$\beta=\min\{\alpha-1, 1\}$ for $\alpha>1$.

If $\alpha>1.5$, then $\beta\in(2^{-1}, 1]$.
From Theorem \ref{cor1} and Theorem \ref{cor1b},  for any $r > 0$, we have
\beq\label{remb2}
\lim_{n\rightarrow\infty}\frac{1}{b_n^2}\log\pp
\left(\frac{\sqrt{n}}{b_n\sigma(\alpha-1)h_{\alpha,0}}|\hat h_{\alpha,0}-h_{\alpha,0}|>r\right)= -\frac{r^2}{2}
\deq
where
 $$
\sigma^2=\left(\frac{\alpha}{(\alpha-1)h_{\alpha,0}}\right)^2\left[\sum_{i=1}^{\infty}p_{i}^{2\alpha-1}
-\left(\sum_{i=1}^{\infty}p_{i}^\alpha\right)^2\right],
$$
and the moderate deviation scale $\{b_n, n\ge 1\}$ is a sequence of positive numbers satisfying
$$
b_n\rightarrow\infty\ \ \text{and}\ \ \frac{b_n}{\sqrt{n}}\to 0.
$$
Since $\sqrt{n}/b_n\to\infty$, then from (\ref{remb2}), for any $r>0$, we have
$$
\lim_{n\rightarrow\infty}\frac{1}{b_n^2}\log\pp
\left(|\hat h_{\alpha,0}-h_{\alpha,0}|>r\right)= -\infty
$$
and
$$
\lim_{n\rightarrow\infty}\frac{1}{b_n^2}\log\pp
\left(\frac{\sqrt{n}}{2b_n\sigma(\alpha-1)h_{\alpha,0}^2}(\hat h_{\alpha,0}-h_{\alpha,0})^2>r\right)= -\infty,
$$
which implies
$$
\lim_{n\rightarrow\infty}\frac{1}{b_n^2}\log\pp
\left(\frac{\sqrt{n}}{b_n}|R_{h_{\alpha,0}}(\hat h_{\alpha,0})|>r\right)= -\infty.
$$
Hence from (\ref{remb1}) and (\ref{remb2}), we have
$$
\lim_{n\rightarrow\infty}\frac{1}{b_n^2}\log\pp
\left(\frac{\sqrt{n}}{b_n\sigma}|\hat{\mathcal{R}}_{\alpha,n}-\mathcal{R}_\alpha|>r\right)= -\frac{r^2}{2}.
$$
\end{exa}

\begin{exa}Let the index $h_{\alpha,0}$ be defined in (\ref{h}). For $\alpha>1$, consider Hill diversity number
$$
\mathcal{N}_\alpha:=(h_{\alpha,0})^{\frac{1}{1-\alpha}}=\left(\sum_{k=1}^\infty p_k^\alpha\right)^{\frac{1}{1-\alpha}}
$$
and its plug-in estimator
$$
\hat{\mathcal{N}}_{\alpha,n}:=(\hat h_{\alpha,0})^{\frac{1}{1-\alpha}}=\left(\sum_{k=1}^\infty \hat p_k^\alpha\right)^{\frac{1}{1-\alpha}}.
$$
By using Taylor's formula, we have
$$
\left(\hat h_{\alpha,0}\right)^{\frac{1}{1-\alpha}}=( h_{\alpha,0})^{\frac{1}{1-\alpha}}+\frac{1}{1-\alpha}(h_{\alpha,0})^{\frac{\alpha}{1-\alpha}}(\hat h_{\alpha,0}-h_{\alpha,0})+R_{h_{\alpha,0}}(\hat h_{\alpha,0})
$$
where
$$
R_{h_{\alpha,0}}(\hat h_{\alpha,0}):=\frac{\alpha}{2(1-\alpha)^2}\xi^{\frac{2\alpha-1}{1-\alpha}}(\hat h_{\alpha,0}-h_{\alpha,0})^2
$$
and $\xi$ is between $\hat h_{\alpha,0}$ and $h_{\alpha,0}$. Hence we get
\beq\label{remn1}
\hat{\mathcal{N}}_{\alpha,n}-\mathcal{N}_\alpha=\frac{1}{1-\alpha}(h_{\alpha,0})^{\frac{\alpha}{1-\alpha}}(\hat h_{\alpha,0}-h_{\alpha,0})+R_{h_{\alpha,0}}(\hat h_{\alpha,0}).
\deq
Let $g(x)=(1-\alpha)^{-1}(h_{\alpha,0})^{\frac{\alpha}{1-\alpha}}x^{\alpha}$ and
$$
\hat\theta_n-\theta:=\sum_{k=1}^{\infty} g(\hat p_k)-\sum_{k=1}^{\infty} g(p_k)=\frac{1}{1-\alpha}(h_{\alpha,0})^{\frac{\alpha}{1-\alpha}}(\hat h_{\alpha,0}-h_{\alpha,0}).
$$
From Proposition \ref{prop1}, $g'$ is $\beta$-H\"older continuous with
$\beta=\min\{\alpha-1, 1\}$ for $\alpha>1$.

If $\alpha>1.5$, then $\beta\in(2^{-1}, 1]$.
From Theorem \ref{cor1} and Theorem \ref{cor1b},  for any $r > 0$, we have
\beq\label{remn2}
\lim_{n\rightarrow\infty}\frac{1}{b_n^2}\log\pp
\left(\frac{\sqrt{n}}{b_n\sigma(\alpha-1)}(h_{\alpha,0})^{\frac{\alpha}{1-\alpha}}|\hat h_{\alpha,0}-h_{\alpha,0}|>r\right)= -\frac{r^2}{2}
\deq
where
 $$
\sigma^2=\left(\frac{\alpha}{(\alpha-1)}(h_{\alpha,0})^{\frac{\alpha}{1-\alpha}}\right)^2\left[\sum_{k=1}^{\infty}p_{k}^{2\alpha-1}
-\left(\sum_{k=1}^{\infty}p_{k}^\alpha\right)^2\right],
$$
and the moderate deviation scale $\{b_n, n\ge 1\}$ is a sequence of positive numbers satisfying
$$
b_n\rightarrow\infty\ \ \text{and}\ \ \frac{b_n}{\sqrt{n}}\to 0.
$$
Since $\sqrt{n}/b_n\to\infty$, then from (\ref{remn2}), for any $r>0$, we have
$$
\lim_{n\rightarrow\infty}\frac{1}{b_n^2}\log\pp
\left(|\hat h_{\alpha,0}-h_{\alpha,0}|>r\right)= -\infty
$$
and
$$
\lim_{n\rightarrow\infty}\frac{1}{b_n^2}\log\pp
\left(\frac{\sqrt{n}}{2b_n\sigma(\alpha-1)^2}(h_{\alpha,0})^{\frac{2\alpha}{1-\alpha}}(\hat h_{\alpha,0}-h_{\alpha,0})^2>r\right)= -\infty,
$$
which implies
$$
\lim_{n\rightarrow\infty}\frac{1}{b_n^2}\log\pp
\left(\frac{\sqrt{n}}{b_n}|R_{h_{\alpha,0}}(\hat h_{\alpha,0})|>r\right)= -\infty.
$$
Hence from (\ref{remn1}) and (\ref{remn2}), we have
$$
\lim_{n\rightarrow\infty}\frac{1}{b_n^2}\log\pp
\left(\frac{\sqrt{n}}{b_n\sigma}|\hat{\mathcal{N}}_{\alpha,n}-\mathcal{N}_\alpha|>r\right)= -\frac{r^2}{2}.
$$
\end{exa}

\section{Proofs of main results}
We state some useful lemmas to prove these main results.
\begin{lem}\label{lemM} \cite[Lemma 6.1]{G-Z}
If $g: [0,1] \rightarrow \rr$ is differentiable on $[0,1]$ and its derivative $g'$ is $\beta$-H\"older continuous, then for any $a \in (0,1]$ we can write
$$ g(x)=g(a)+g'(a)(x-a)+R_a(x),$$
where
$$|R_a(x)|\leq M|x-a|^{\beta+1}$$
for some $M > 0$.
\end{lem}

\begin{lem}\label{lem-hr} \cite[Theorem 3.1]{H-R}
Let $X_1, X_2, \ldots, X_n$ be independent random variables defined on a probability $(\Omega, \FF, \pp)$.
Let us consider for all integer $n\ge 2$,
$$
U_n=\sum_{i=2}^n \sum_{j=1}^{i-1} g_{i,j}(X_i, X_j),
$$
where the $g_{i,j}: \rr\times\rr\to \rr$ are Borel measurable functions verifying
$$
\ee\left(g_{i,j}(X_i, X_j)|X_i\right)=0\ \ \text{and}\ \ \ee\left(g_{i,j}(X_i, X_j)|X_j\right)=0.
$$
Let $u>0$, $\varepsilon>0$ and let $|g_{i,j}|\le A$ for all $i,j$. Then we have
$$
\aligned
&\pp\left[ U_n\ge (1+\varepsilon)C\sqrt{2u}+\left(2\sqrt{\kappa} D+\frac{1+\varepsilon}{3} F\right)u \right.\\
& \ \ \ \ \ \ \ \ \ \ \ \ \ \ \ \ \ \
\left.
+\left(\sqrt{2}\kappa(\varepsilon)+\frac{2\sqrt{\kappa}}{3}\right)Bu^{3/2}+\frac{\kappa(\varepsilon)}{3} Au^2\right] \\
\le & 3e^{-u} \wedge 1.
\endaligned
$$
Here
\beq\label{C}
C^2=\sum_{i=2}^n \sum_{j=1}^{i-1}\ee\left(g_{i,j}^2(X_i, X_j)\right),
\deq
\beq\label{D}
\aligned
D=&\sup\left\{\ee\left(\sum_{i=2}^n\sum_{j=1}^{i-1}g_{i,j}(X_i, X_j)a_i(X_i)b_j(X_j)\right): \right. \\
& \ \ \ \ \ \ \ \ \ \ \ \ \ \ \ \ \ \  \ \ \ \ \ \ \ \
\left. \ee\left(\sum_{i=2}^na_i^2(X_i)\right)\le 1, \ \ \ee\left(\sum_{j=1}^{n-1}b_j^2(X_j)\right)\le 1
\right\},
\endaligned
\deq
\beq\label{F}
F=\ee\left(\sup_{i,t}\left|\sum_{j=1}^{i-1}g_{i,j}(t, X_j)\right|\right),
\deq
\beq\label{B}
B^2=\max\left\{\sup_{i,t}\left(\sum_{j=1}^{i-1}\ee\left(g_{i,j}^2(t, X_j)|X_i=t\right)\right),\
\sup_{j,t}\left(\sum_{i=j+1}^{n}\ee\left(g_{i,j}^2(X_i, t)|X_j=t\right)\right)
\right\},
\deq
where $\kappa$ and $\kappa(\varepsilon)$ can be chosen respectively equal to $4$ and $(2.5+32\varepsilon^{-1})$.
\end{lem}

\begin{proof} [{\bf Proof of Theorem \ref{thm1}}]
From Lemma \ref{lemM}, we have
\beq\label{thm-main}
\hat \theta_n-\theta_n=\sum_{i=1}^\infty g'(p_{n,i})(\hat p_{n,i}-p_{n,i})+\sum_{i=1}^\infty R_{p_{n,i}}(\hat p_{n,i}).
\deq
For every $n\ge 1$ and $1\le k\le n$, let us define
$$
T_{k,n}:=\sum_{i=1}^{\infty}(I_{\{X_{k, n}=a_i\}}-p_{n,i})g'(p_{n,i}),
$$
then we have
$$
\sum_{i=1}^\infty g'(p_{n,i})(\hat p_{n,i}-p_{n,i})=\frac{1}{n}\sum_{k=1}^nT_{k,n}
$$
and
$$
\aligned
Var(T_{k,n})=&\ee\left(\sum_{i=1}^{\infty}I_{\{X_{k, n}=a_i\}}g'(p_{n,i})\right)^2-\left(\sum_{i=1}^{\infty}p_{n,i}g'(p_{n,i})\right)^2\\
=&\sum_{i=1}^{\infty}p_{n,i}\left(g'(p_{n,i})\right)^2-\left(\sum_{i=1}^{\infty}p_{n,i}g'(p_{n,i})\right)^2=\sigma_n^2.
\endaligned
$$
Since $g'$ is Lipschitz continuous, there exists a positive constant $M$, such that
\beq\label{g}
|T_{k,n}|\le M\sum_{i=1}^{\infty}(I_{\{X_{k, n}=i\}}+p_{n,i})\le 2M.
\deq
In order to prove Theorem \ref{thm1}, it is enough to show the following claims: for any $r> 0$,
\beq\label{thm-main1}
\aligned
\lim_{n\rightarrow\infty}\frac{1}{b_n^2}\log\pp\left(\frac{\sqrt{n}}{b_n\sigma_n}\left|\sum_{i=1}^{\infty}(\hat p_{n,i}-p_{n,i})g'(p_{n,i})\right|>r\right)=-\frac{r^2}{2}
\endaligned
\deq
and for any $\varepsilon > 0$,
\beq\label{thm-main2}
\lim_{n\rightarrow\infty}\frac{1}{b_n^2}\log\pp
\left(\frac{\sqrt{n}}{b_n\sigma_n}\left|\sum_{i=1}^\infty R_{p_{n,i}}(\hat p_{n,i})\right|>\varepsilon\right)= -\infty.
\deq

{\bf Proof of the claim (\ref{thm-main1}).}  By using G\"artner-Ellis Theorem (see \cite{2}), we have only to prove that the following limit holds: for any $\lambda\in\rr$,
\beq\label{thm-main1-1}
\lim_{n\rightarrow\infty}\frac{1}{b_n^2}\log\ee
\exp\left(\frac{\lambda b_n}{\sqrt{n}\sigma_n}\sum_{k=1}^nT_{k,n}\right)=\frac{\lambda^2}{2}.
\deq
From the fact that $T_{1,n}$ is bounded, and the condition $\frac{b_n}{\sqrt{n}\sigma_n^3}\to 0$ and the following elementary inequality
$$
\left|e^x-1-x-\frac{x^2}{2}\right|\le \frac{|x|^3}{3!}e^{|x|}\ \ \text{for}\  x\in\rr,
$$
we get
\beq\label{thm-main1-3}
\aligned
\left|\ee
\exp\left(\frac{\lambda b_n}{\sqrt{n}\sigma_n}T_{1,n}\right)-1
-\frac{\lambda^2b_n^2}{2n}\right|
\le & \ee\left(\frac{|\lambda|^3 b_n^3}{3!\sqrt{n}^3\sigma_n^3}|T_{1,n}|^3 e^{\frac{b_n}{\sqrt{n}\sigma_n}|\lambda T_{1,n}|}\right)\\
\le& C_{1,\lambda} \frac{b_n^3}{\sqrt{n}^3\sigma_n}
\endaligned
\deq
where $C_{1,\lambda}$ is a positive constant dependent on $\lambda$. Furthermore, since $\frac{b_n}{\sqrt{n}\sigma_n}\to 0$, then we have
$$
\frac{b_n^3}{\sqrt{n}^3\sigma_n}=o\left(\frac{b_n^2}{n}\right),
$$
which implies
$$
\ee\exp\left(\frac{\lambda b_n}{\sqrt{n}\sigma_n}T_{1,n}\right)=1+\frac{\lambda^2b_n^2}{2n}+o\left(\frac{b_n^2}{n}\right).
$$
Hence we can get
\begin{align*}
\lim_{n\rightarrow\infty}\frac{1}{b_n^2}\log\ee
\exp\left\{\frac{\lambda b_n}{\sqrt{n}\sigma_n}\sum_{k=1}^nT_{k,n}\right\}
=&\lim_{n\rightarrow\infty}\frac{n}{b_n^2}\log\ee
\exp\left\{\frac{\lambda b_n}{\sqrt{n}\sigma_n}T_{1,n}\right\}\\
=&\lim_{n\rightarrow\infty}\frac{n}{b_n^2}\log\left(1+\frac{ \lambda^2 b_n^2}{2n}+o\left(\frac{b_n^2}{n}\right)\right)
=\frac{\lambda^2}{2},
\end{align*}
which is the claim (\ref{thm-main1}).

{\bf Proof of the claim (\ref{thm-main2}).}
From Lemma \ref{lemM}, we have
\beq\label{thm-r}
\aligned
\left|\sum_{i=1}^\infty R_{p_{n,i}}(\hat p_{n,i})\right|
\leq & M\sum_{i=1}^\infty(\hat p_{n,i}-p_{n,i})^2\\
=&\frac{M}{n^2}\sum_{i=1}^\infty\sum_{k=1}^n\left(I_{\{X_{k, n}=a_i\}}-p_{n,i}\right)^2\\
&\ \ \ \  +\frac{M}{n^2}\sum_{i=1}^\infty\sum_{k\ne l}^n\left(I_{\{X_{k, n}=a_i\}}-p_{n,i}\right)\left(I_{\{X_{l, n}=a_i\}}-p_{n,i}\right).
\endaligned
\deq
Hence, in order to prove (\ref{thm-main2}), it is enough to show that the following claims hold: for any $\varepsilon > 0$,
\beq\label{thm-r1}
\lim_{n\to\infty}\frac{1}{b_n^2}\log \pp\left(\frac{1}{\sigma_n b_n n^{3/2}}\left|\sum_{i=1}^\infty\sum_{k=2}^n \sum_{l=1}^{k-1} \left(I_{\{X_{k, n}=a_i\}}-p_{n,i}\right)\left(I_{\{X_{l, n}=a_i\}}-p_{n,i}\right)\right|>\varepsilon\right)=-\infty
\deq
and
\beq\label{thm-r2}
\lim_{n\to\infty}\frac{1}{b_n^2}\log \pp\left(\frac{1}{\sigma_n b_n n^{3/2}}\sum_{i=1}^\infty\sum_{k=1}^n\left(I_{\{X_{k, n}=a_i\}}-p_{n,i}\right)^2>\varepsilon\right)=-\infty.
\deq

Firstly, for every $n\ge 1$, let us define
$$
U_n=\sum_{k=2}^n \sum_{l=1}^{k-1} g_{k,l}(X_{k,n}, X_{l,n}),
$$
where
$$
g_{k,l}(X_{k,n}, X_{l,n})=\sum_{i=1}^\infty \left(I_{\{X_{k, n}=a_i\}}-p_{n,i}\right)\left(I_{\{X_{l,n}=a_i\}}-p_{n,i}\right).
$$
It is easy to check that $|g_{k,l}(X_{k,n}, X_{l,n})|\le 2$ and
$$
\ee\left(g_{k,l}(X_{k,n}, X_{l,n})|X_{k,n}\right)= \ee\left(g_{k,l}(X_{k,n}, X_{l,n})|X_{l,n}\right)=0,\ \ \ \ \ k\ne l.
$$
Now we shall estimate the parameters $C, D, F, B$ in Lemma \ref{lem-hr}. From the boundedness of $|g_{k,l}(X_{k,n}, X_{l,n})|$, we have
$$
C^2=\sum_{k=2}^n \sum_{l=1}^{k-1}\ee\left(g_{k,l}^2(X_{k,n}, X_{l,n})\right)\le 4n^2,
$$
$$
F=\ee\left(\sup_{k,t}\left|\sum_{l=1}^{k-1}g_{k,l}(t, X_{l,n})\right|\right)\le 2n
$$
and
$$
B^2\le 4n.
$$
Furthermore, under the conditions $\ee\left(\sum_{k=2}^na_k^2(X_{k,n})\right)\le 1$ and $\ee\left(\sum_{l=1}^{n-1}b_l^2(X_{l,n})\right)\le 1$, by using H\"{o}lder's inequality and Jensen's inequality, we have
$$
\aligned
&\left|\ee\left(\sum_{k=2}^n\sum_{l=1}^{k-1}g_{k,l}(X_{k,n}, X_{l,n})a_k(X_{k,n})b_l(X_{l,n})\right)\right|\\
=&\left|\ee\left(\sum_{k=2}^n\sum_{l=1}^{k-1}\sum_{i=1}^\infty \left(I_{\{X_{k, n}=a_i\}}-p_{n,i}\right)\left(I_{\{X_{l, n}=a_i\}}-p_{n,i}\right)a_k(X_{k,n})b_l(X_{l,n})\right)\right|\\
\le & \sum_{k=2}^n\sum_{l=1}^{k-1}2\ee|a_k(X_{k,n})|\ee|b_l(X_{l,n})|\\
\le & 2\sum_{k=2}^n\left(\ee a_k^2(X_{k,n})\right)^{1/2}\sum_{l=1}^{n-1}\left(\ee b_l^2(X_{l,n})\right)^{1/2}\\
\le & 2n\left(\left(\sum_{k=2}^n\ee a_k^2(X_{k,n})\right)\cdot\left(\sum_{l=1}^{n-1}\ee b_l^2(X_{l,n})\right)\right)^{1/2}\le 2n,
\endaligned
$$
which implies $D\le 2n$.

Let us define
\beq\label{delta}
\aligned
\Delta_n:=&(1+\varepsilon)C\sqrt{2u_n}+\left(2\sqrt{\kappa} D+\frac{1+\varepsilon}{3} F\right)u_n\\
&\ \ \ +\left(\sqrt{2}\kappa(\varepsilon)+\frac{2\sqrt{\kappa}}{3}\right)Bu^{3/2}_n+\frac{\kappa(\varepsilon)}{3} Au^2_n.
\endaligned
\deq
Since $g'$ is Lipschitz continuous, then we have
$$
\sigma^2_n=\sum_{i=1}^{\infty}p_{n,i}\left(g'\left( p_{n,i}\right)\right)^2-\left(\sum_{i=1}^{\infty}p_{n,i}g'\left( p_{n,i}\right)\right)^2\le M^2
$$
where $M$ is defined in (\ref{g}). From the condition $\frac{b_n}{\sqrt{n}\sigma_n}\to 0$,  we can choose a sequence of positive numbers $\{l_n, n\ge 1\}$ such that
$$
l_n\to\infty\ \ \text{and}\ \ \frac{\sqrt{n}\sigma_n}{l_nb_n}\to\infty.
$$
By taking the sequence $u_n=b_n \sqrt{n} \sigma_n/l_n$ in (\ref{delta}), we get
$$
\aligned
\Delta_n=O\left(n^{5/4}\sqrt{\frac{b_n\sigma_n}{l_n}}+\frac{b_n\sigma_nn^{3/2}}{l_n}+n^{5/4}\left(\frac{b_n\sigma_n}{l_n}\right)^{3/2}+\frac{b_n^2n\sigma_n^2}{l_n^2}\right).
\endaligned
$$
Moreover, from the condition $\frac{b_n}{\sqrt{n}\sigma_n}\to 0$, it is easy to check
$$
\frac{b_n\sigma_n n^{3/2}}{n^{5/4}\sqrt{b_n\sigma_n/l_n}}=n^{1/4}\sqrt{b_n\sigma_nl_n}=\left(\frac{\sqrt{n}\sigma_n}{b_n}b_n^2l_n\right)^{1/2}\to \infty,
$$
$$
\frac{b_n\sigma_n n^{3/2}}{(b_n\sigma_n/l_n)^{3/2}n^{5/4}}=\frac{n^{1/4}l_n^{3/2}}{\sqrt{b_n\sigma_n}}=\left(\frac{\sqrt{n}\sigma_n}{b_n}\frac{l_n^3}{\sigma_n^2}\right)^{1/2}\to \infty
$$
and
$$
\frac{b_n\sigma_n n^{3/2}}{b_n^2n\sigma_n^2 /l_n^2}=\frac{\sqrt{n}l_n^2}{b_n\sigma_n}=\frac{\sqrt{n}\sigma_n}{b_n}\frac{l_n^2}{\sigma_n^2}\to \infty,
$$
which yields that
$$
\aligned
\Delta_n=o\left(b_n\sigma_nn^{3/2}\right).
\endaligned
$$
Therefore, by using Lemma \ref{lem-hr}, for any $\varepsilon>0$, we have
$$
\aligned
\lim_{n\to\infty}\frac{1}{b_n^2}&\log \pp\left(\frac{1}{\sigma_n b_n n^{3/2}}\left|\sum_{i=1}^\infty\sum_{k=2}^n \sum_{l=1}^{k-1} \left(I_{\{X_{k, n}=a_i\}}-p_{n,i}\right)\left(I_{\{X_{l, n}=a_i\}}-p_{n,i}\right)\right|>\varepsilon\right)\\
=&\lim_{n\to\infty}\frac{1}{b_n^2}\log \pp\left(\left|U_n\right|>\sigma_n b_n n^{3/2}\varepsilon\right)\\
\le &\lim_{n\to\infty}\frac{1}{b_n^2}\log \pp\left(\left|U_n\right|>\Delta_n\right)\\
\le &-\lim_{n\to\infty}\frac{u_n}{b_n^2}\to-\infty,
\endaligned
$$
which is the claim (\ref{thm-r1}).

Next, for each $k$, since
$$
\sum_{i=1}^\infty\left(I_{\{X_{k, n}=a_i\}}-p_{n,i}\right)^2=\sum_{i=1}^\infty I_{\{X_{k, n}=a_i\}}-2\sum_{i=1}^\infty I_{\{X_{k, n}=a_i\}}p_{n,i}+\sum_{i=1}^\infty p_{n,i}^2\le 2,
$$
then for any $\varepsilon > 0$ and all $n$ large enough, we have
\begin{align*}
&\pp\left(\frac{1}{\sigma_n b_n n^{3/2}}\sum_{i=1}^\infty\sum_{k=1}^n\left(1_{\{X_{k, n}=i\}}-p_{n,i}\right)^2>\varepsilon\right)\\
\le& \sum_{k=1}^n\pp\left(\frac{1}{\sigma_n b_n \sqrt{n}}\sum_{i=1}^\infty\left(1_{\{X_{k, n}=i\}}-p_{n,i}\right)^2>\varepsilon\right)=0
\end{align*}
 which implies the claim (\ref{thm-r2}).

 Based on the above discussions, Theorem \ref{thm1} can be obtained.
\end{proof}

\begin{lem}\label{lemL}
Let $f$ be a Lipschitz continuous function in $[0,1]$. Assume that
$$
b_n\rightarrow\infty\ \ \text{and}\ \ \frac{b_n}{\sqrt{n}}\to 0,
$$
then for any $\varepsilon > 0$,
\beq\label{lemL-1}
\lim_{n\to\infty}\frac{1}{b_n^2}\log \pp\left(\left|\sum_{i=1}^\infty \hat p_{n,i}f(\hat p_{n,i})-\sum_{i=1}^\infty  p_{n,i}f(p_{n,i})\right|>\varepsilon\right)=-\infty.
\deq
\end{lem}
\begin{proof}
Firstly, we have
\begin{align*}
&\sum_{i=1}^\infty \hat p_{n,i}f(\hat p_{n,i})-\sum_{i=1}^\infty  p_{n,i}f(p_{n,i})\\
=& \sum_{i=1}^\infty (\hat p_{n,i}-p_{n,i})f(p_{n,i})+ \sum_{i=1}^\infty (\hat p_{n,i}-p_{n,i})(f(\hat p_{n,i})-f(p_{n,i}))\\
&\ \ \ \ \ \ +\sum_{i=1}^\infty p_{n,i}(f(\hat p_{n,i})-f(p_{n,i})).
\end{align*}
From the condition $\frac{\sqrt{n}}{b_n\sigma_n}\to \infty$ and by using similar proof of (\ref{thm-main1}),  it is not difficult to see that
\beq\label{lemL-2}
\frac{1}{b_n^2}\log \pp\left(\left|\sum_{i=1}^\infty (\hat p_{n,i}-p_{n,i})f(p_{n,i})\right|>\varepsilon\right)\to -\infty.
\deq
Similarly, from (\ref{thm-main2}), we have
\beq\label{lemL-3}
\aligned
&\frac{1}{b_n^2}\log \pp\left(\left| \sum_{i=1}^\infty (\hat p_{n,i}-p_{n,i})(f(\hat p_{n,i})-f(p_{n,i}))\right|>\varepsilon\right)\\
\le &\frac{1}{b_n^2}\log \pp\left(M\left| \sum_{i=1}^\infty (\hat p_{n,i}-p_{n,i})^2\right|>\varepsilon\right)\to-\infty.
\endaligned
\deq
Furthermore, by using Cauchy-Schwarz inequality and (\ref{lemL-3}), we get
\beq\label{lemL-4}
\aligned
&\frac{1}{b_n^2}\log \pp\left(\left|\sum_{i=1}^\infty p_{n,i}(f(\hat p_{n,i})-f(p_{n,i}))\right|>\varepsilon\right)\\
\le & \frac{1}{b_n^2}\log \pp\left(M\sum_{i=1}^\infty p_{n,i}|\hat p_{n,i}-p_{n,i}|>\varepsilon\right)\\
\le & \frac{1}{b_n^2}\log \pp\left(M\sqrt{\sum_{i=1}^\infty p_{n,i}^2}\sqrt{\sum_{i=1}^\infty|\hat p_{n,i}-p_{n,i}|^2}>\varepsilon\right)\\
\le & \frac{1}{b_n^2}\log \pp\left(M^2\sum_{i=1}^\infty(\hat p_{n,i}-p_{n,i})^2>\varepsilon\right) \to -\infty.
\endaligned
\deq
Based on the discussions, the desired result can be obtained.
\end{proof}

\begin{proof} [{\bf Proof of Corollary \ref{cor2}}]
For any $0<\varepsilon<r\wedge1$, we have
\begin{align*}
&\pp
\left(\frac{\sqrt{n}}{b_n\hat \sigma_n}|\hat{\theta}_n-\theta_n|>r\right)\\
=&\pp
\left(\frac{\sqrt{n}}{b_n\hat \sigma_n}|\hat{\theta}_n-\theta_n|>r,\  \left|\frac{\hat\sigma_n^2}{\sigma_n^2}-1\right|\le \varepsilon\right)+\pp
\left(\frac{\sqrt{n}}{b_n\hat \sigma_n}|\hat{\theta}_n-\theta_n|>r,\  \left|\frac{\hat\sigma_n^2}{\sigma_n^2}-1\right|>\varepsilon\right)\\
\le &\pp
\left(\frac{\sqrt{n}}{b_n \sigma_n}|\hat{\theta}_n-\theta_n|>r\sqrt{1-\varepsilon}\right)+\pp\left(\left|\frac{\hat\sigma_n^2}{\sigma_n^2}-1\right|> \varepsilon\right)
\end{align*}
and
\begin{align*}
&\pp
\left(\frac{\sqrt{n}}{b_n\hat \sigma_n}|\hat{\theta}_n-\theta_n|>r\right)\\
\ge &\pp
\left(\frac{\sqrt{n}}{b_n\hat \sigma_n}|\hat{\theta}_n-\theta_n|>r,\  \left|\frac{\hat\sigma_n^2}{\sigma_n^2}-1\right|\le \varepsilon\right)\\
\ge &\pp
\left(\frac{\sqrt{n}}{b_n \sigma_n}|\hat{\theta}_n-\theta_n|>r(1+\varepsilon),\  \left|\frac{\hat\sigma_n^2}{\sigma_n^2}-1\right|\le \varepsilon\right)\\
\ge &\pp
\left(\frac{\sqrt{n}}{b_n \sigma_n}|\hat{\theta}_n-\theta_n|>r\sqrt{1+\varepsilon}\right)-\pp\left(\left|\frac{\hat\sigma_n^2}{\sigma_n^2}-1\right|> \varepsilon\right).
\end{align*}
Firstly, we shall prove the following claim:
\beq\label{cor-r1}
\lim_{n\to\infty}\frac{1}{b_n^2}\log \pp\left(\left|\frac{\hat\sigma_n^2}{\sigma_n^2}-1\right|> \varepsilon\right)=-\infty.
\deq
Since $g'$ is Lipschitz continuous, there exists a positive constant $M$, such that
$$
\left|\sum_{i=1}^{\infty}\hat p_{n,i}g'\left( \hat p_{n,i}\right)\right|\le M \ \ \text{and} \ \ \left|\sum_{i=1}^{\infty}p_{n,i}g'\left( p_{n,i}\right)\right|\le M.
$$
Moreover, there exists a positive constant $K$, such that for any $x, y\in[0,1]$,
$$
\left|(g'(x))^2-(g'(y))^2\right|\le \left|g'(x)+g'(y)\right|\left|g'(x)-g'(y)\right|\le 2KM|x-y|,
$$
namely, $(g')^2$ is also Lipschitz continuous.
Hence we have
\beq\label{cor-r2}
\aligned
\left|\hat\sigma_n^2-\sigma_n^2\right|\le&\left|\sum_{i=1}^{\infty}\hat p_{n,i}\left(g'\left( \hat p_{n,i}\right)\right)^2-\sum_{i=1}^{\infty}p_{n,i}\left(g'\left( p_{n,i}\right)\right)^2\right|\\
&\ \ \ \ +\left|\left(\sum_{i=1}^{\infty}p_{n,i}g'\left( p_{n,i}\right)\right)^2
-\left(\sum_{i=1}^{\infty}\hat p_{n,i}g'\left( \hat p_{n,i}\right)\right)^2\right|\\
\le & \left|\sum_{i=1}^{\infty}\hat p_{n,i}\left(g'\left( \hat p_{n,i}\right)\right)^2-\sum_{i=1}^{\infty}p_{n,i}\left(g'\left( p_{n,i}\right)\right)^2\right|\\
&\ \ \ \ +2M\left|\sum_{i=1}^{\infty}\hat p_{n,i}g'\left( \hat p_{n,i}\right)-\sum_{i=1}^{\infty}p_{n,i}g'\left( p_{n,i}\right)\right|.
\endaligned
\deq
From Lemma \ref{lemL}, for any $\varepsilon > 0$, we have
$$\frac{1}{b_n^2}\log \pp\left(\left|\sum_{i=1}^\infty \hat p_{n,i}g'(\hat p_{n,i})-\sum_{i=1}^\infty  p_{n,i}g'(p_{n,i})\right|>\varepsilon\right)\to-\infty$$
and
$$\frac{1}{b_n^2}\log \pp\left(\left|\sum_{i=1}^\infty \hat p_{n,i}(g'(\hat p_{n,i}))^2-\sum_{i=1}^\infty  p_{n,i}(g'(p_{n,i}))^2\right|>\varepsilon\right)\to-\infty,$$
which, together with (\ref{cor-r2}), implies that
\beq\label{cor-r3}
\frac{1}{b_n^2}\log \pp\left(\left|\hat\sigma_n^2-\sigma_n^2\right|>\varepsilon\right)\to-\infty.
\deq
From the condition $\liminf_{n\to\infty} \sigma_n^2>0$, (\ref{cor-r3}) and the following relation
$$
\frac{\hat\sigma_n^2}{\sigma_n^2}-1=\frac{\hat\sigma_n^2-\sigma_n^2}{\sigma_n^2},
$$
the claim (\ref{cor-r1}) holds.

From Theorem \ref{thm1}, we have
 $$
\limsup_{n\to\infty}\frac{1}{b_n^2}\log \pp\left(\frac{\sqrt{n}}{b_n\hat \sigma_n}|\hat{\theta}_n-\theta_n|>r\right)\le -\frac{r^2(1-\varepsilon)}{2}
 $$
and
$$
\liminf_{n\to\infty}\frac{1}{b_n^2}\log \pp\left(\frac{\sqrt{n}}{b_n\hat \sigma_n}|\hat{\theta}_n-\theta_n|>r\right)\ge -\frac{r^2(1+\varepsilon)}{2}.
 $$
By the arbitrariness of $\varepsilon$, we can get
 $$
\lim_{n\to\infty}\frac{1}{b_n^2}\log \pp\left(\frac{\sqrt{n}}{b_n\hat \sigma_n}|\hat{\theta}_n-\theta_n|>r\right)= -\frac{r^2}{2}.
 $$
\end{proof}
\begin{proof} [{\bf Proof of Theorem \ref{cor1}}]
Note that for random variables with nonuniform distribution, obviously we have $\sigma^2 > 0$. Theorem $\ref{cor1}$ is a special case of Theorem $\ref{thm1}$ and the proof is totally similar to that of Theorem $\ref{thm1}$.
\end{proof}

\begin{proof} [{\bf Proof of Theorem \ref{thm2}}]
From Lemma \ref{lemM}, we have
\beq\label{thm-main-1}
\hat \theta_n-\theta_n=\sum_{i=1}^\infty g'(p_{n,i})(\hat p_{n,i}-p_{n,i})+\sum_{i=1}^\infty R_{p_{n,i}}(\hat p_{n,i}).
\deq
By the similar proof of Theorem \ref{thm1}, it is enough to show that
for any $\varepsilon > 0$,
\beq\label{thm-main2-2}
\lim_{n\rightarrow\infty}\frac{1}{b_n^2}\log\pp
\left(\frac{\sqrt{n}}{b_n\sigma_n}\left|\sum_{i=1}^\infty R_{p_{n,i}}(\hat p_{n,i})\right|>\varepsilon\right)= -\infty.
\deq
From Lemma \ref{lemM} and H\"older's inequality, we have
\beq\label{thm-r-2}
\aligned
\left|\sum_{i=1}^\infty R_{p_{n,i}}(\hat p_{n,i})\right|
\leq & M\sum_{i=1}^\infty|\hat p_{n,i}-p_{n,i}|^{\beta+1}=M\sum_{i=1}^\infty|\hat p_{n,i}-p_{n,i}|^{2\beta}|\hat p_{n,i}-p_{n,i}|^{1-\beta}\\
\le & M\left(\sum_{i=1}^\infty|\hat p_{n,i}-p_{n,i}|^{2}\right)^{\beta}\left(\sum_{i=1}^\infty|\hat p_{n,i}-p_{n,i}|\right)^{1-\beta}\\
\le & 2^{1-\beta}M\left(\sum_{i=1}^\infty|\hat p_{n,i}-p_{n,i}|^{2}\right)^{\beta},
\endaligned
\deq
which implies
$$
\aligned
\pp\left(\frac{\sqrt{n}}{b_n\sigma_n}\left|\sum_{i=1}^\infty R_{p_{n,i}}(\hat p_{n,i})\right|>\varepsilon\right)
\le\pp\left(\left(\frac{\sqrt{n}}{b_n\sigma_n}\right)^{1/\beta}\sum_{i=1}^\infty|\hat p_{n,i}-p_{n,i}|^{2}>\left(\frac{\varepsilon}{2^{1-\beta}M}\right)^{1/\beta}\right).
\endaligned
$$
Hence, in order to prove (\ref{thm-main2-2}), it is enough to show that the following claims hold: for any $\varepsilon > 0$,
\beq\label{thm-r1-2}
\lim_{n\to\infty}\frac{1}{b_n^2}\log \pp\left(\frac{1}{(\sigma_n b_n)^{1/\beta}\sqrt{n}^{4-1/\beta}}\left|\sum_{i=1}^\infty\sum_{k=2}^n \sum_{l=1}^{k-1} \left(I_{\{X_{k, n}=a_i\}}-p_{n,i}\right)\left(I_{\{X_{l, n}=a_i\}}-p_{n,i}\right)\right|>\varepsilon\right)=-\infty
\deq
and
\beq\label{thm-r2-3}
\lim_{n\to\infty}\frac{1}{b_n^2}\log \pp\left(\frac{1}{(\sigma_n b_n)^{1/\beta}\sqrt{n}^{4-1/\beta}}\sum_{i=1}^\infty\sum_{k=1}^n\left(I_{\{X_{k, n}=a_i\}}-p_{n,i}\right)^2>\varepsilon\right)=-\infty.
\deq
From the condition $\frac{b_n}{\sqrt{n}\sigma_n^{1/(2\beta-1)}}\to 0$,  we can choose a sequence of positive numbers $\{l_n, n\ge 1\}$ such that
$$
l_n\to\infty\ \ \text{and}\ \ \frac{\sqrt{n}\sigma_n^{1/(2\beta-1)}}{l_n^{\beta/(2\beta-1)}b_n}\to\infty.
$$
As the similar proof as (\ref{thm-r1}), by taking the sequence $u_n=b_n^2l_n$ in (\ref{delta}), we get
$$
\aligned
\Delta_n=O\left(n\sqrt{b_n^2l_n}+nb_n^2l_n+\sqrt{n}(b_n^2l_n)^{3/2}+(b_n^2l_n)^2\right).
\endaligned
$$
 Because of $\beta\in(2^{-1},1)$, it is easy to check
$$
\aligned
\frac{(\sigma_n b_n)^{1/\beta}\sqrt{n}^{4-1/\beta}}{n\sqrt{b_n^2l_n}}
=b_n\sqrt{l_n}\left(\frac{\sqrt{n}\sigma_n^{1/(2\beta-1)}}{l_n^{\beta/(2\beta-1)}b_n}\right)^{2-1/\beta}\to\infty,
\endaligned
$$
$$
\frac{(\sigma_n b_n)^{1/\beta}\sqrt{n}^{4-1/\beta}}{nb_n^2l_n}
=\left(\frac{\sqrt{n}\sigma_n^{1/(2\beta-1)}}{l_n^{\beta/(2\beta-1)}b_n}\right)^{2-1/\beta}\to\infty,
$$
$$
\frac{(\sigma_n b_n)^{1/\beta}\sqrt{n}^{4-1/\beta}}{\sqrt{n}(b_n^2l_n)^{3/2}}
=\frac{\sqrt{n}}{b_n\sqrt{l_n}}\left(\frac{\sqrt{n}\sigma_n^{1/(2\beta-1)}}{l_n^{\beta/(2\beta-1)}b_n}\right)^{2-1/\beta}\to\infty\ \ \ \text{by}\ \frac{\beta}{2\beta-1}>\frac{1}{2}
$$
and
$$
\frac{(\sigma_n b_n)^{1/\beta}\sqrt{n}^{4-1/\beta}}{(b_n^2l_n)^2}=\left(\frac{\sqrt{n}}{b_n\sqrt{l_n}}\right)^2\left(\frac{\sqrt{n}\sigma_n^{1/(2\beta-1)}}{l_n^{\beta/(2\beta-1)}b_n}\right)^{2-1/\beta}\to\infty
$$
which yields that
$$
\aligned
\Delta_n=o\left((\sigma_n b_n)^{1/\beta}\sqrt{n}^{4-1/\beta}\right).
\endaligned
$$
Therefore, by using Lemma \ref{lem-hr}, the claim (\ref{thm-r1-2}) holds. By using the proof of (\ref{thm-r2}), the claim (\ref{thm-r2-3}) holds.
\end{proof}

\begin{lem}\label{lemb}
Let $f$ be a $\beta$-H\"older continuous function in $[0,1]$ for some $\beta\in(2^{-1},1)$. Assume that
$$
b_n\rightarrow\infty\ \ \text{and}\ \ \frac{\sqrt{n}\sigma_n^{1/(2\beta-1)}}{b_n}\to\infty,
$$
then for any $\varepsilon > 0$,
\beq\label{lemb-1}
\lim_{n\to\infty}\frac{1}{b_n^2}\log \pp\left(\left|\sum_{i=1}^\infty \hat p_{n,i}f(\hat p_{n,i})-\sum_{i=1}^\infty  p_{n,i}f(p_{n,i})\right|>\varepsilon\right)=-\infty.
\deq
\end{lem}
\begin{proof}
Firstly, we have
\begin{align*}
&\sum_{i=1}^\infty \hat p_{n,i}f(\hat p_{n,i})-\sum_{i=1}^\infty  p_{n,i}f(p_{n,i})\\
=& \sum_{i=1}^\infty (\hat p_{n,i}-p_{n,i})f(p_{n,i})+ \sum_{i=1}^\infty (\hat p_{n,i}-p_{n,i})(f(\hat p_{n,i})-f(p_{n,i}))\\
&\ \ \ \ \ \ +\sum_{i=1}^\infty p_{n,i}(f(\hat p_{n,i})-f(p_{n,i})).
\end{align*}
From
$$
\frac{\sqrt{n}}{b_n\sigma_n}=\frac{\sqrt{n}\sigma_n^{1/(2\beta-1)}}{b_n}\frac{1}{\sigma_n^{2\beta/(2\beta-1)}},
$$
then we have $\frac{\sqrt{n}}{b_n\sigma_n}\to \infty$.
By using similar proof of (\ref{thm-main1}),  it is not difficult to see that
\beq\label{lemb-2}
\frac{1}{b_n^2}\log \pp\left(\left|\sum_{i=1}^\infty (\hat p_{n,i}-p_{n,i})f(p_{n,i})\right|>\varepsilon\right)\to -\infty.
\deq
Similarly, from (\ref{thm-main2-2}), we have
\beq\label{lemb-3}
\aligned
&\frac{1}{b_n^2}\log \pp\left(\left| \sum_{i=1}^\infty (\hat p_{n,i}-p_{n,i})(f(\hat p_{n,i})-f(p_{n,i}))\right|>\varepsilon\right)\\
\le &\frac{1}{b_n^2}\log \pp\left(2^{1-\beta}M\left(\sum_{i=1}^\infty|\hat p_{n,i}-p_{n,i}|^{2}\right)^{\beta}>\varepsilon\right)\to-\infty.
\endaligned
\deq
Furthermore, by using H\"older's inequality, we get
$$
\aligned
\sum_{i=1}^\infty p_{n,i}|f(\hat p_{n,i})-f(p_{n,i})|\le & M\sum_{i=1}^\infty p_{n,i}|\hat p_{n,i}-p_{n,i}|^\beta\\
\le & M\left(\sum_{i=1}^\infty p_{n,i}^{2/(2-\beta)}\right)^{(2-\beta)/2}\left(\sum_{i=1}^\infty |\hat p_{n,i}-p_{n,i}|^2\right)^{\beta/2}
\endaligned
$$
which, together with (\ref{lemb-3}), implies
\beq\label{lemb-4}
\aligned
&\frac{1}{b_n^2}\log \pp\left(\left|\sum_{i=1}^\infty p_{n,i}(f(\hat p_{n,i})-f(p_{n,i}))\right|>\varepsilon\right)\\
\le & \frac{1}{b_n^2}\log \pp\left(M^2\left(\sum_{i=1}^\infty |\hat p_{n,i}-p_{n,i}|^2\right)^{\beta}>\varepsilon^2\right) \to -\infty.
\endaligned
\deq
Based on the discussions, the desired result can be obtained.
\end{proof}
\begin{proof}[{\bf Proof of Corollary \ref{cor2b}}]
Since $g'$ is $\beta$-H\"older continuous, $g'$ is also bounded, i.e., there exists a positive constant $M$, such that
$$
\left|\sum_{i=1}^{\infty}\hat p_{n,i}g'\left( \hat p_{n,i}\right)\right|\le M \ \ \text{and} \ \ \left|\sum_{i=1}^{\infty}p_{n,i}g'\left( p_{n,i}\right)\right|\le M.
$$
Moreover, there exists a positive constant $K$, such that for any $x, y\in[0,1]$,
$$
\left|(g'(x))^2-(g'(y))^2\right|\le \left|g'(x)+g'(y)\right|\left|g'(x)-g'(y)\right|\le 2KM|x-y|^\beta,
$$
namely, $(g')^2$ is also $\beta$-H\"older continuous. Hence, by using similar proof as Corollary \ref{cor2}, the desired result can be obtained.
\end{proof}
\begin{proof}[{\bf Proof of Theorem \ref{cor1b}}] The proof is similar as Theorem \ref{cor1}.
\end{proof}

\section*{Disclosure statement}
No potential conflict of interest was reported by the authors.

\end{document}